\documentclass{iasr}
\renewcommand\thisnumber{0}
\renewcommand\thisyear {2009}
\renewcommand\thismonth{June}
\renewcommand\thisvolume{00}
\usepackage[dvips]{graphicx}
\usepackage{amsthm}

\begin{document}

\title{On the Order of  Polynilpotent Multipliers of Some Nilpotent Products of Cyclic $p$-Groups }
\author[Behrooz Mashayekhy  and Fahimeh Mohammadzadeh]{Behrooz Mashayekhy\affil{1}\comma\corrauth,
      Fahimeh Mohammadzadeh\affil{2}}
 \address{\affilnum{1}Department of Pure Mathematics, Center of Excellence in Analysis on \linebreak Algebraic Structures, Ferdowsi University of Mashhad, Mashhad, Iran. \\
           \affilnum{2}Department of Mathematics,
           Payame Noor University, Ahvaz, Iran.}
 \corraddr{Behrooz Mashayekhy, Department of Pure Mathematics, Center of Excellence in Analysis on Algebraic Structures,
Ferdowsi University of Mashhad, P.O. Box 1159-91775, Mashhad, Iran.
          Email: \tt mashaf@math.um.ac.ir}
\received{29 April~ 2009}

\newtheorem{thm}{Theorem}[section]
 \newtheorem{cor}[thm]{Corollary}
 \newtheorem{lem}[thm]{Lemma}
 \newtheorem{prop}[thm]{Proposition}
 \newtheorem{defn}[thm]{Definition}
 \newtheorem{rem}[thm]{Remark}
\newtheorem{Example}{Example}[section]

\begin{abstract}
In this article we show that if ${\cal V}$ is  the variety of
polynilpotent groups of class row $(c_1,c_2,...,c_s),\  {\mathcal
N}_{c_1,c_2,...,c_s}$, and  $G\cong{\bf
{Z}}_{p^{\alpha_1}}\stackrel{n}{*}{\bf
{Z}}_{p^{\alpha_2}}\stackrel{n}{*}...\stackrel{n}{*}{\bf{Z}}_{p^{\alpha_t} }$ is the $n$th nilpotent product of some cyclic $p$-groups, where
$c_1\geq n$, $\alpha_1 \geq \alpha_2 \geq...\geq \alpha_t $ and $
(q,p)=1$ for all primes $q$ less than or equal to $n$, then
$|{\mathcal N}_{c_1,c_2,...,c_s}M(G)|=p^{d_m}$ if and only if
$G\cong{\bf {Z}}_{p}\stackrel{n}{*}{\bf
{Z}}_{p}\stackrel{n}{*}...\stackrel{n}{*}{\bf{Z}}_{p
}$ ($m$-copies), where $m=\sum _{i=1}^t \alpha_i$ and
$d_m=\chi_{c_s+1}(...(\chi_{c_2+1}(\sum_{j=1}^n
\chi_{c_1+j}(m)))...)$. Also, we extend the result to the multiple nilpotent product  $G\cong{\bf
{Z}}_{p^{\alpha_1}}\stackrel{n_1}{*}{\bf
{Z}}_{p^{\alpha_2}}\stackrel{n_2}{*}...\stackrel{n_{t-1}}{*}{\bf{Z}}_{p^{\alpha_t} }$, where $c_1\geq n_1\geq...\geq n_{t-1}$. Finally a similar result is given for the $c$-nilpotent
multiplier of $G\cong{\bf
{Z}}_{p^{\alpha_1}}\stackrel{n}{*}{\bf
{Z}}_{p^{\alpha_2}}\stackrel{n}{*}...\stackrel{n}{*}{\bf{Z}}_{p^{\alpha_t}}$ with the different conditions $n \geq c$ and $ (q,p)=1$ for all primes $q$ less than or
equal to $n+c.$
\end{abstract}

\maketitle

\section{Introduction and Motivation}
\label{sec1}
Let $G$ be any group with a presentation $G\cong F/R$, where F is
a free group. Then the Baer invariant of $G$ with respect to the
variety of groups $\mathcal{V}$, denoted by ${\mathcal V}M(G)$,
is defined to be
$${\mathcal V}M(G)=\frac{R \cap
V(F)}{[RV^{*}F]},$$ where $V$ is the set of words of the variety
$\mathcal{V}$, $V(F)$ is the verbal subgroup of $F$ and $$ [R
V^{*}F]=\langle v(f_1,...,f_{i-1},f_{i}r,f_{i+1},...,f_n)
 v(f_1,...,f_{i},...,f_n)^{-1} |$$ $$ r\in R, f_i\in F, v
\in V, 1\leq i\leq n, n \in \mathbf{N} \rangle.$$

One may check that
${\mathcal V}M(G)$ is abelian and independent of the choice of
the free presentation of $G$. In particular, if $\mathcal{V}$ is
the variety of abelian groups, $\mathcal{A}$, then the Baer
invariant of the group $G$ will be $ (R \cap F')/[R, F],$ which is
isomorphic to the well-known notion the Schur multiplier of $G$,
denoted by $M(G)$. If $\mathcal{V}$ is the variety of
polynilpotent groups of class row $(c_1,...,c_s)$, $ {\mathcal
N}_{c_1,c_2,...,c_s}$, then the Baer invariant of a group $G$
with respect to this variety, which is called a polynilpotent multiplier of $G$, is as follows:
$${\mathcal N}_{c_1,c_2,...,c_s} M(G)=\frac{R \cap \gamma_{c_s+1}\circ
...\circ \gamma_{c_1+1}(F)}{[R,\ _{c_1}F,\
_{c_2}\gamma_{c_1+1}(F),...,\ _{c_s}\gamma_{c_{s-1}+1}\circ
...\circ \gamma_{c_1+1}(F)]},$$ where $\gamma_{c_s+1}\circ
...\circ \gamma_{c_1+1}(F)=\gamma_{c_s+1}(\gamma_{c_{s-1}+1}( ...
(\gamma_{c_1+1}(F))...))$ are the term of iterated lower central
series of $F$. See Hekster [6] for the equality $$[R {\mathcal
N^{*}}_{c_1,c_2,...,c_s}F]=[R,\ _{c_1}F,\
_{c_2}\gamma_{c_1+1}(F),...,\ _{c_s}\gamma_{c_{s-1}+1}\circ
...\circ \gamma_{c_1+1}(F)].$$  In particular, if $s=1$ and
$c_1=c$, then the Baer invariant of $G$ with respect to the variety
${\mathcal N}_c$, which is called the $c$-nilpotent multiplier of
$G$, is $${\mathcal N}_cM(G)\cong\frac{R \cap \gamma _{c+1}(F)}{[R,\
_cF]}.$$

Historically, Green [4] showed that the order of the Schur
multiplier of a finite $p$-group of order $p^{n}$ is bounded by
$p^{\frac{n(n-1)}{2}}$. Berkovich [2] showed that a finite
$p$-group of order $p^{n}$ is an elementary abelian $p$-group if
and only if the order of $M(G)$ is $p^{n(n-1)/2}$. Moghaddam
[15,16] presented a bound for the polynilpotent multiplier of a
finite $p$-group. He showed that if  ${\mathcal V}$ is the
variety of polynilpotent groups of a given class row and $ G $ is
a finite $d$-generator group of order $p^n$, then
$$|{\mathcal V}M( {\mathbf Z}_{p}^{(d)})| \leq |{\mathcal
V}M(G)||V(G)|\leq|{\mathcal V}M({\mathbf Z}_{p}^{(n)})|,$$
 where
${\mathbf Z}_{n}^{(m)}$ denotes the direct sum of $m$ copies of
${\mathbf Z}_{n}$. As a consequence, using the structure of
${\mathcal V}M(\textbf{Z}_{p}^{(n)})$ in [12], we can show that the order of the nilpotent multiplier of
a finite $p$-group of order $p^n$ is bounded by
$p^{\chi_{c+1}(n)}$, where $\chi_{c+1}(n) $ is the number of
basic commutators of weight $c+1$ on $ n $ letters. The first
author and  Sanati [13] extended a result of  Berkovich to
the $c$-nilpotent multiplier of a finite $p$-group. They showed
that for an abelian $p$-group $G$, $|{\mathcal
N}_cM(G)|=p^{\chi_{c+1}(n)}$ if and only if  $G$ is an elementary
abelian $p$-group. Putting an additional condition on the kernel of
the left natural map of the generalized Stallings-Stammbach
five-term exact sequence, they showed that an arbitrary finite
$p$-group with the $c$-nilpotent multiplier of maximum order is
an elementary abelian $p$-group.

Unfortunately, there is a mistake in the proof of Theorem 3.5 in
[13] due to using the inequality
$i\chi_{c+1}(i)<\chi_{c+1}(i+1)$ which is not correct in general.
In this paper, first, we give a correct proof for Theorem 3.5
in [13]. Second, we extend the result in
different directions. In fact, we show that if ${\mathcal V}$ is the
variety of polynilpotent groups of class row $(c_1,c_2,...,c_s)$,
${\mathcal N}_{c_1,c_2,...,c_s}$, and  $G\cong{\bf
{Z}}_{p^{\alpha_1}}\stackrel{n}{*}{\bf
{Z}}_{p^{\alpha_2}}\stackrel{n}{*}...\stackrel{n}{*}{\bf{Z}}_{p^{\alpha_t}}$ is the $n$th nilpotent product of some cyclic $p$-groups,
where $c_1\geq n$, $\alpha_1 \geq \alpha_2 \geq...\geq \alpha_t $
and $ (q,p)=1$ for all primes $q$ less than or equal to $n$,  then
$|{\mathcal N}_{c_1,c_2,...,c_s}M(G)|=p^{d_m}$ if and only if
$G\cong{\bf {Z}}_{p}\stackrel{n}{*}{\bf
{Z}}_{p}\stackrel{n}{*}...\stackrel{n}{*}{\bf{Z}}_{p
}$ ($m$-copies), where $m=\sum _{i=1}^t \alpha_i$ and
$d_m=\chi_{c_s+1}(...(\chi_{c_2+1}(\sum_{j=1}^n
\chi_{c_1+j}(m)))...)$. Also, we extend the above result to the
multiple nilpotent product of cyclic $p$-groups  $G\cong{\bf
{Z}}_{p^{\alpha_1}}\stackrel{n_1}{*}{\bf
{Z}}_{p^{\alpha_2}}\stackrel{n_2}{*}...\stackrel{n_{t-1}}{*}{\bf{Z}}_{p^{\alpha_t}}$, when $c_1\geq n_1\geq...\geq n_{t-1}$. As a consequence we
show that the polynilpotent multiplier of a finite abelian
$p$-group $G$ has maximum order if and only if $G$ is an
elementary abelian $p$-group. Finally we give a similar result for
the $c$-nilpotent multiplier of $G\cong{\bf
{Z}}_{p^{\alpha_1}}\stackrel{n}{*}{\bf
{Z}}_{p^{\alpha_2}}\stackrel{n}{*}...\stackrel{n}{*}{\bf{Z}}_{p^{\alpha_t}}$, with the different conditions $n \geq c$  and $(q,p)=1$
for all primes $q$ less than or equal to $n+c.$

\section{Notation and Preliminaries}
\label{sec2}
\begin{defn} Let  $ \{ G_{i}\} _{i \in I } $ be
a family of arbitrary groups. The $ n $th nilpotent product of
the family $ \{ G_{i}\} _{i \in I } $ is defined as follows: $$
\prod ^{\stackrel{n}{*}} _{i \in I }G_{i} = \frac{ \prod ^{* }
_{i \in I } G_{i}}{ \gamma _{n+1}(\prod  ^{* } _{i \in I }G_{i})
\cap [G_{i}] ^{*} _{i \in I } },$$ where $ \prod ^{* } _{i \in I
}G_{i} $ is the free product of the family $ \{ G_{i} \}_{ i \in
I } $, and $$ [G_{i}] ^{*} _{i \in I } = \langle [ G_{i}, G_{j}]
| i,j \in I, i \neq j \rangle ^{ \prod ^{* } _{i \in I }G_{i}} $$
is the cartesian subgroup of the free product $ \prod  ^{* } _{i
\in I }G_{i} $ which is the kernel of the natural homomorphism
from $ \prod  ^{* } _{i \in I }G_{i} $ to the direct product $
\prod ^{\times } _{i \in I }G_{i} $. For further properties of
the above notation see Neumann [17]. If $ \{ G_{i}\} _{i \in I }
$ is a family of cyclic groups, then  $ \gamma _{n+1}(\prod  ^{* }
_{i \in I } G_{i}) \subseteq [G_{i}] ^{*}  $ and hence $ \prod
^{\stackrel{n}{*}} _{i \in I }G_{i} = \prod ^{* } _{i \in I }
G_{i} / \gamma _{n+1}(\prod  ^{* } _{i \in I
}G_{i})$.
\end{defn}
\begin{defn} A variety $\mathcal{V}$ is said to
be a Schur-Baer variety if  for any group $G$ for which the
marginal factor group $G/V^*(G)$ is finite, then the verbal
subgroup $V(G)$ is also finite and $|V(G)|$ divides a power of
$|G/V^*(G)|$.
\end{defn}

 Schur  proved that the variety of abelian groups, $\mathcal{A}$,
is a Schur-Baer variety (see [10]). Also, Baer [1] proved that if $u$ and $v$ have Schur-Baer property, then
the variety defined by the word $[u,v]$ has the above property.

The following theorem gives a very important property of
Schur-Baer varieties.
\begin{thm} (Leedham-Green, McKay [11]).
The following conditions on a variety
$\mathcal{V}$ are
equivalent: \\
(i) $\mathcal{V}$ is a Schur-Baer variety.\\
(ii) For every finite group $ G $, its Baer invariant,
${\mathcal V}M(G)$, is of order dividing a power of $|G|$.
\end{thm}

In the rest of this section we review some theorems required in the proofs of the main results of the article.
\begin{thm} (Jones [9]). Let $G $ be a
finite $ d $-generator group of order $p^{n}$. Then
$$p^{\frac{1}{2}d(d-1)} \leq |G'||M(G)| \leq p^{\frac{1}{2}n(n-1)}.$$ 
\end{thm}
\begin{thm} (Berkovich [2]). Let $G $
be a finite group of order $p^{n}$. Then $|M(G)|=
p^{\frac{1}{2}n(n-1)}$ if and only if  $G$ is
an elementary abelian $p$-group.
\end{thm}
\begin{thm} (Moghaddam [15,16]). Let
${\mathcal V}$ be the variety of polynilpotent groups of a given
class row. Let $ G $ be a finite $ d $-generator group of order
$p^n$. Then
$$|{\mathcal V}M( \textbf{Z}_{p}^{(d)})| \leq |{\mathcal
V}M(G)||V(G)|\leq |{\mathcal V}M(\textbf{Z}_{p}^{(n)})|.$$
\end{thm}

We recall that the number of basic commutators of weight $c$ on
$n$ generators, denoted by $\chi _{c}(n)$, is determined by Witt
formula (see [5]).
\begin{thm} (Moghaddam and Mashayekhy [14]).
Let $G \cong \textbf{Z}_{n_1}\oplus
\textbf{Z}_{n_2}\oplus ...\oplus \textbf{Z}_{n_k}$ be a finite
abelian groups, where $n_{i+1}$ divides $n_i$ for all $1\leq i\leq
k-1$. Then for all $c \geq 1$, the $ c $-nilpotent multiplier of
$ G $ is
$${\mathcal N}_cM(G)= \textbf{Z}_{n_2}^{(b_2)} \oplus \textbf{Z}_{n_3}^{(b_3-b_2)} \oplus ... \oplus
\textbf{Z}_{n_k}^{(b_{k}-b_{k-1})},$$
where $b_i=\chi_{c+1}(i)$.
\end{thm}

The following result is an interesting
consequence of Theorems 2.6 and 2.7.
\begin{cor} Let $G$ be a finite
$d$-generator $p$-group of order $p^n$, then
$$ p^{\chi_{c+1}(d)}\leq |{\mathcal N}_cM(G)||\gamma_{c+1}(G)|\leq
p^{\chi_{c+1}(n)}.$$
\end{cor}

The following theorem is a generalization of Theorem 2.5.
\begin{thm} (Mashayekhy, Sanati [13]).
Let $G$ be an abelian group of order $p^n$. Then
${\mathcal N}_cM(G)=p^{\chi_{c+1}(n)}$ if and only if $G$ is an
elementary abelian $p$-group.
\end{thm}

Let $1
\rightarrow N \rightarrow G \rightarrow Q \rightarrow 1$ be a
${\cal V}$-central extension, where ${\cal V}$ is any variety of
groups, that is, the above sequence is exact and $N$ is
contained in the marginal subgroup of $G$, $V^{*}(G)$. Then the
following five-term exact sequence exists (see Fr\"{o}hlich [3]):
 $${\mathcal V}M(G)  \stackrel {\theta}{\rightarrow} {\mathcal V}M(Q) \rightarrow N \rightarrow G/V(G) \rightarrow Q/V(Q) \rightarrow
 1.$$

The nonabelian version of Theorem 2.9 is as follows.
\begin{thm} (Mashayekhy, Sanati [13]).
Let $G$ be a finite $p$-group of order $p^n$. If
${\mathcal N}_cM(G)=p^{\chi_{c+1}(n)}$,
then\\
(i) There is an epimorphism $ {\mathcal N}_cM(G) \stackrel
{\theta}{\rightarrow} {\mathcal N}_cM(G/G')$ which is obtained
from the
Fr\"{o}hlich sequence.\\
(ii) If $ker (\theta )=1$, then $G$ is an elementary abelian
$p$-group.
\end{thm}
\begin{thm} (Mashayekhy, Parvizi [12]).
Let $$G \cong \textbf{Z}^{(m)} \oplus
\textbf{Z}_{n_1}\oplus \textbf{Z}_{n_2}\oplus ...\oplus
\textbf{Z}_{n_k}$$ be a finitely generated abelian group, where
$n_{i+1}$ divides $n_i$ for all $1\leq i\leq k-1$. Then $${\mathcal
N}_{c_1,c_2,...,c_s} M(G)\cong \textbf{Z}^{(\beta_m)} \oplus
\textbf{Z}_{n_1}^{(\beta_{m+1}-\beta_m)} \oplus ... \oplus
\textbf{Z}_{n_k}^{(\beta_{m+k}-\beta_{m+k-1})},$$ where $\beta_i=\chi
_{c_s+1}( \chi _{c_{s-1}+1}(...(\chi _{c_1+1}(i))...)) $ for all
$m \leq i \leq m+k$.
\end{thm}

Theorems 2.6  and  2.11 imply the following useful
inequalities.
\begin{cor} With the notation of previous theorem let $G$ be a finite
$d$-generator $p$-group of order $p^n$. Then
$$p^{\beta_d} \leq |{\mathcal
N}_{c_1,c_2,...,c_s} M(G)||\gamma_{c_s+1}(\gamma_{c_{s-1}+1}( ...
(\gamma_{c_1+1}(G))...))|\leq p^{\beta_n}.$$
\end{cor}
\begin{thm} (Hokmabadi, Mashayekhy, Mohammadzadeh [8]).
Let $G\cong \underbrace{{\bf
{Z}}\stackrel{n}{*}...\stackrel{n}{*}{\bf
{Z}}}_{m-copies}\stackrel{n}{*} {\bf
{Z}}_{r_1}\stackrel{n}{*}...\stackrel{n}{*}{\bf{Z}}_{r_t}$, be the
$n$th nilpotent product of some cyclic groups, where $r_{i+1}$
divides $r_i$ for all $ 1 \leq i \leq t-1$. If $c \geq n$ and
$(p,r_1)=1$ for all primes $p$ less than or equal to $n$, then the
$c$-nilpotent multiplier of $G$ is isomorphic to
$${\bf{Z}}^{(\sum ^{n}_{i=1}\chi _{c+i}(m))} \oplus {\bf{Z}}_{r_1}^{(\sum ^{n}_{i=1}
(\chi _{c+i}(m+1)-\chi _{c+i}(m)))}\oplus ...\oplus
{\bf{Z}}_{r_t}^{(\sum ^{n}_{i=1}(\chi _{c+i}(m+t)-\chi
_{c+i}(m+t-1)))}.$$
\end{thm}
\begin{thm} (Hokmabadi, Mashayekhy,
Mohammadzadeh [8]). Let $G\cong\underbrace{{\bf
{Z}}\stackrel{n}{*}...\stackrel{n}{*}{\bf
{Z}}}_{m-copies}\stackrel{n}{*} {\bf
{Z}}_{r_1}\stackrel{n}{*}...\stackrel{n}{*}{\bf{Z}}_{r_t}$ be the
$n$th nilpotent product of some cyclic groups, where $r_{i+1}$
divides $r_i$ for all $ 1 \leq i \leq t-1$. If $(p,r_1)=1$ for all
primes $p$ less than or equal to $n$, then the structure of the
polynilpotent multiplier of $G$ is
$${\mathcal N}_{c_1,c_2,...,c_s}M(G)= {\bf{Z}}^{(d_m)} \oplus {\bf{Z}}_{r_1}^{(d_{m+1}-d_m)}
\oplus ...\oplus{\bf{ Z}}_{r_t}^{ (d_{m+t}-d_{m+t-1})},$$  where
$d_i=\chi_{c_s+1}(...(\chi_{c_2+1}(\sum _{j=1}^{n}\chi
_{c_1+j}(i)))...),$ for all $c_1 \geq n$
 and $c_2, ..., c_s \geq 1$ and $ m \leq i \leq m+t$.
\end{thm}
\begin{thm} (Hokmabadi, Mashayekhy [7]).
Let $G\cong\underbrace{{\bf
{Z}}\stackrel{n}{*}...\stackrel{n}{*}{\bf
{Z}}}_{m-copies}\stackrel{n}{*} {\bf
{Z}}_{r_1}\stackrel{n}{*}...\stackrel{n}{*}{\bf{Z}}_{r_t}$ be the
$n$th nilpotent product of some cyclic groups such that
$r_{i+1}$ divides $r_i$ for all $ 1 \leq i \leq t-1$. If
$(p,r_1)=1$ for any prime $p$ less than or equal to $n+c$, then\\
(i) if $n\geq c$, then ${\mathcal N}_{c}M(G)= {\bf{Z}}^{(g_0)}
\oplus {\bf{Z}}_{r_1}^{(g_{1}-g_0)} \oplus ...\oplus{\bf{
Z}}_{r_t}^{ (g_{t}-g_{t-1})};$\\
(ii)  if $c \geq n$, then ${\mathcal N}_{c}M(G)= {\bf{Z}}^{(f_0)}
\oplus {\bf{Z}}_{r_1}^{(f_{1}-f_0)} \oplus ...\oplus{\bf{
Z}}_{r_t}^{ (f_{t}-f_{t-1})},$\\ where $f_k=\sum _{i=1}^{n}\chi
_{c+i}(m+k)$ and $g_k=\sum _{i=1}^{c}\chi _{n+i}(m+k)$  for $0
\leq k \leq t.$
\end{thm}
\begin{thm} (Hokmabadi, Mashayekhy,
Mohammadzadeh [8]). Let $G\cong A_1 \stackrel{n_1}{*}
A_2 \stackrel{n_2}{*} ... \stackrel{n_k}{*} A_{k+1}$ such that
$A_i\cong {\bf Z}$ for $1\leq i \leq t$ and $A_j\cong {\bf
Z}_{m_j}$ for  $ t+1 \leq j \leq k+1$. Let $c_1 \geq n_1 \geq n_2
\geq ...\geq n_k$ and $m_{k+1}| m_k|...|m_{t+1}$ and
$(p,m_{t+1})=1$ for all primes $p \leq n_1$. Then the structure of
the polynilpotent multiplier of $G$ is
$${\mathcal N}_{c_1,c_2,...,c_s}M(G)= {\bf{Z}}^{(e_0)} \oplus {\bf{Z}}_{m_{t+1}}^{(e_{t}-e_o)}
\oplus ...\oplus{\bf{ Z}}_{m_{k+1}}^{ (e_{k}-e_{k-1})},$$  where
$e_i=\chi_{c_s+1}(...(\chi_{c_2+1}(u+\sum ^{i}_{j=t}h_j))...),$
for all $ t \leq i \leq k$,
$e_0=\chi_{c_s+1}(...\\ (\chi_{c_2+1}(u))...)$, $u = \sum _{j=1}
^{n_{t-1}}\chi _{c_1+j}(t)+\sum _{i=1} ^{{t-2}} \sum
_{j=n_{i+1}+1} ^{n_{i}}\chi _{c_1+j}(i+1)$ and $h_j=\sum _{\lambda=1}
^{n_{j}}(\chi _{c_1+\lambda}(j+1)-\chi _{c_1+\lambda}(j))$.
\end{thm}

\section{Main Results}
\label{sec3}
As we mentioned before, there is a mistake in the proof of
Theorem 2.9. More precisely, in the proof of Theorem 3.5 in
[13] it is assumed that $G$ is a finite
abelian $d$-generator $p$-group of order $p^n$ and $|{\mathcal
N}_cM(G)|=p^{\chi _{c+1}(n)}$. Then using the inequality
 $i\chi _{c+1}(i)< \chi _{c+1}(i+1)$ it is proved that
$n=d$ and therefore $G$ is an elementary abelian $p$-group. Unfortunately, the inequality
$i\chi _{c+1}(i)< \chi _{c+1}(i+1)$ is not correct and so the proof is not valid.

 In this section, first, we intend to present a new proof for Theorem 2.9 in
  order to remedy the above mentioned mistake. Second, using this
   new method, we extend the result to polynilpotent multipliers of nilpotent
products of cyclic $p$-groups with some conditions.\\

\hspace{-0.5cm}\textbf{Proof of Theorem 2.9}.

\begin{proof} Let $G$ be an elementary
abelian $p$-group of order $p^n$. Then
by Theorem 2.7  we have ${\mathcal N}_cM(G)=\textbf{Z}_{p}^{(b_2)}
\oplus \textbf{Z}_{p}^{(b_3-b_2)}
\oplus ... \oplus \textbf{Z}_{p}^{(b_n-b_{n-1})}$, where $b_i= \chi _{c+1}(i)$, and hence  $|{\mathcal N}_cM(G)|=p^{\chi_{c+1}(n)}$.

Conversely, suppose  that $|{\mathcal N}_cM(G)|=p^{\chi_{c+1}(n)}$. Since
$G$ is an abelian $p$-group of order $p^n$, we can consider $G$ as follows:
$$G\cong
 \textbf{Z}_{p^{\alpha_1}}\oplus
\textbf{Z}_{p^{\alpha_2}}\oplus ...\oplus
\textbf{Z}_{p^{\alpha_d}},$$
where $\alpha_1 \geq \alpha_2\geq
...\geq \alpha_d$ and $\alpha_1 + \alpha_2+...\alpha_d=n$.
By Theorem 2.7 ${\mathcal
N}_cM(G)=\textbf{Z}_{p^{\alpha_2}}^{(b_2)} \oplus
\textbf{Z}_{p^{\alpha_3}}^{(b_3-b_2)} \oplus ... \oplus
\textbf{Z}_{p^{\alpha_d}}^{(b_d-b_{d-1})}$
and so  $|{\mathcal N}_cM(G)|=p^{\alpha_2 b_2+\alpha_3
(b_3-b_2)+...+\alpha_d (b_d-b_{d-1})}$. On the other hand by
hypothesis $|{\mathcal N}_cM(G)|=p^{b_n}$. Therefore
 $b_n=\alpha_2
b_2+\alpha_3 (b_3-b_2)+...+\alpha_d (b_d-b_{d-1})$. Also
$b_n=(b_n-b_{n-1})+(b_{n-1}-b_{n-2})+...+(b_3-b_2)+b_2$. Thus
$$(b_n-b_{n-1})+(b_{n-1}-b_{n-2})+...+(b_3-b_2)+b_2=\alpha_2 b_2+\alpha_3
(b_3-b_2)+...+\alpha_d (b_d-b_{d-1})=$$
$$\underbrace{b_2+...+b_2}_{\alpha_2-copies}+
\underbrace{(b_3-b_2)+...+(b_3-b_2)}_{\alpha_3-copies}+...+
\underbrace{(b_d-b_{d-1})+...+(b_d-b_{d-1})}_{\alpha_d-copies}.$$
So we have the following equality:
$$(b_n-b_{n-1})+(b_{n-1}-b_{n-2})+...+(b_{d+1}-b_d) =$$
$$\underbrace{b_2+...+b_2}_{\alpha_2-1-copies}+
\underbrace{(b_3-b_2)+...+(b_3-b_2)}_{\alpha_3-1-copies}+...
+\underbrace{(b_d-b_{d-1})+...+(b_d-b_{d-1})}_{\alpha_d-1-copies}\ (I).$$

One can easily see that for any $i\geq 1,\ (b_i-b_{i-1})$ is the number
of basic commutators of weight $c+1$ on $i$ letters such that
$x_i$ does appear in it. So $(b_j-b_{j-1}) \geq (b_i-b_{i-1})$ whenever
$j \geq i.$ Now, assume $\alpha_1 \geq 2$. Then $ n-1>n-\alpha_1$ and so the left-hand side
of the above equality has more terms than the right-hand side. Also each term of the left-hand side of the above equality is greater than of any term of the right-hand side.
These facts imply that the equality $(I)$ dose
not hold which is a contradiction. Thus we must have $\alpha_1=1$ and hence
$d=n,$ $\alpha_1=\alpha_2=...=\alpha_n=1$. Therefore the result
holds. 
\end{proof}

The next theorem is a generalization of Theorem 2.9. Note that the
nilpotent product of finitely many finite $p$-groups
is also a finite $p$-group.
\begin{thm} Let $G\cong {\bf {Z}}_{p^{\alpha_1}}\stackrel{n}{*}{\bf
{Z}}_{p^{\alpha_2}}\stackrel{n}{*}...\stackrel{n}{*}{\bf{Z}}_{p^{\alpha_t}}$ be the $n$th nilpotent product of some cyclic groups, where
$\alpha_1 \geq \alpha_2 \geq...\geq \alpha_t $ and $ (q,p)=1$ for
all primes $q$ less than or equal to $n$.
Let ${\mathcal
N}_{c_1,c_2,...,c_s}$ be a variety of
polynilpotent groups such that $c_1 \geq n $. Then $|{\mathcal
N}_{c_1,c_2,...,c_s}M(G)|=p^{d_m}$ if and only if
$G\cong\underbrace{{\bf {Z}}_{p}\stackrel{n}{*}{\bf
{Z}}_{p}\stackrel{n}{*}...\stackrel{n}{*}{\bf{Z}}_{p
}}_{m-copies},$ where $m=\sum _{i=1}^t \alpha_i$ and
$d_m=\chi_{c_s+1}(...(\chi_{c_2+1}(\sum_{j=1}^n
\chi_{c_1+j}(m)))...)$.
\end{thm}
\begin{proof} Let $G=\underbrace{{\bf
{Z}}_{p}\stackrel{n}{*}{\bf
{Z}}_{p}\stackrel{n}{*}...\stackrel{n}{*}{\bf{Z}}_{p
}}_{m-copies}$ and $ (q,p)=1$ for all primes $q$ less than or equal to
$n$. Then by Theorem 2.14, $${\mathcal N}_{c_1,c_2,...,c_s}M(G)=
 {\bf{Z}}_{p}^{(d_2)} \oplus
...\oplus{\bf{ Z}}_{p}^{ (d_{m}-d_{m-1})},$$  where
$d_i=\chi_{c_s+1}(...(\chi_{c_2+1}(\sum _{j=1}^{n}\chi
_{c_1+j}(i)))...),$ for all $c_1 \geq n$. Hence $$|{\mathcal
 N}_{c_1,c_2,...,c_s}M(G)|=p^{d_m}.$$

 Conversely, suppose that $|{\mathcal
 N}_{c_1,c_2,...,c_s}M(G)|=p^{d_m}$. By the hypothesis $G={\bf {Z}}_{p^{\alpha_1}}\stackrel{n}{*}{\bf
{Z}}_{p^{\alpha_2}}\stackrel{n}{*}...\stackrel{n}{*}{\bf{Z}}_{p^{\alpha
t} }$ where $\alpha_1 \geq \alpha_2 \geq...\geq \alpha_t $ and
$\alpha_1+ \alpha_2 +...+\alpha_t=m$. Now Theorem 2.14 implies
that $${\mathcal N}_{c_1,c_2,...,c_s}M(G)=
{\bf{Z}}_{p^{\alpha_2}}^{(d_2)} \oplus
{\bf{Z}}_{p^{\alpha_3}}^{(d_3-d_2)}...\oplus{\bf{
Z}}_{p^{\alpha_t}}^{ (d_{t}-d_{t-1})},$$  where
$d_i=\chi_{c_s+1}(...(\chi_{c_2+1}(\sum _{j=1}^{n}\chi
_{c_1+j}(i)))...)$. Thus $$|{\mathcal
N}_{c_1,c_2,...,c_s}M(G)=p^{\alpha_2 d_2+\alpha_3
(d_3-d_2)+...+\alpha_t (d_t-d_{t-1})} .$$ On the other hand by
hypothesis $|{\mathcal
 N}_{c_1,c_2,...,c_s}M(G)|=p^{d_m}$. Therefore\\
 $d_m=\alpha_2
d_2+\alpha_3 (d_3-d_2)+...+\alpha_t (d_t-d_{t-1})$. Now applying
a similar method to the proof of Theorem 2.9, it is enough to
show that if $j\geq i$, then $(d_j-d_{j-1}) \geq (d_i-d_{i-1}).$
In order to prove this fact consider the following sets:
 $$A_1= \{ \alpha |  \alpha \ is\ a\ basic\ commutator\ of\
weight \ c_1+1,..., c_1+n \ on\
 \ x_1,...,x_i\  \} $$ and inductively for all $2 \leq k \leq s$
$$A_k= \{ \alpha | \alpha \ is\ a\ basic\ commutator\ of\
weight \ c_k+1 \  on\ A_{k-1}  \}. $$ Clearly
$d_i=| A_s|$. It is easy to see that $$d_i-d_{i-1}= |\{ \alpha | \alpha \ is\
a\ basic\ commutator\ of\ weight \ c_s+1 \  on$$ $$\ \ \ \ \
A_{s-1}\ such\ that\ x_i\ does \ appear \ in\ \alpha\   \}|.$$  Hence the required inequality holds.
\end{proof}

Using Theorem 2.16 and a similar proof to the above and noting that $j\geq i$ implies $e_j-e_{j-1}\geq e_i-e_{i-1}$,
we can state the following theorem.
\begin{thm} Let $G \cong {\bf {Z}}_{p^{\alpha_1}}\stackrel{n_1}{*}{\bf
{Z}}_{p^{\alpha_2}}\stackrel{n_2}{*}...\stackrel{n_{t-1}}{*}{\bf{Z}}_{p^{\alpha_t}}$ be the $n$th nilpotent product of some cyclic groups, where
$\alpha_1 \geq \alpha_2 \geq...\geq \alpha_t $ and $ (q,p)=1$ for
all primes $q$ less than or equal to $n$.
Let ${\mathcal
N}_{c_1,c_2,...,c_s}$ be a variety of
polynilpotent groups such that $c_1 \geq n_1 $.
Then $|{\mathcal
N}_{c_1,c_2,...,c_s}M(G)|=p^{e_{m-1}}$ if and only if
$G\cong \underbrace{{\bf {Z}}_{p}\stackrel{n_1}{*}{\bf
{Z}}_{p}\stackrel{n_2}{*}...\stackrel{n_{m-1}}{*}{\bf{Z}}_{p
}}_{m-copies}$, where $m=\sum _{i=1}^t \alpha_i$, $e_{m-1}=\chi_{c_s+1}(...(\chi_{c_2+1}(\sum
^{m-1}_{j=0}h_j))...)$,
 and $h_j=\sum _{\lambda=1} ^{n_{j}}(\chi
_{c_1+\lambda}(j+1)-\chi _{c_1+\lambda}(j))$.
\end{thm}

The following result is a consequence of Theorem 2.15 and the above mentioned method
 with different condition $n\geq c$.
\begin{thm} Let $G\cong {\bf
{Z}}_{p^{\alpha_1}}\stackrel{n}{*}{\bf
{Z}}_{p^{\alpha_2}}\stackrel{n}{*}...\stackrel{n}{*}{\bf{Z}}_{p^{\alpha_t}}$ be the $n$th nilpotent product of some cyclic groups, where
$\alpha_1 \geq \alpha_2 \geq...\geq \alpha_t $, $n\geq c$ and $
(q,p)=1$ for all primes $q$ less than or equal to $n+c$. Then
${\mathcal N}_{c}M(G)=p^{g_m}$ if and only if
$G\cong \underbrace{{\bf {Z}}_{p}\stackrel{n}{*}{\bf
{Z}}_{p}\stackrel{n}{*}...\stackrel{n}{*}{\bf{Z}}_{p
}}_{m-copies},$ where $m=\sum _{i=1}^t \alpha_i$ and
$g_m = \sum_{i=1}^{c} \chi_{n+i}(m).$
\end{thm}

With the assumption and notation of Theorem 3.1, let $n=1$. Then
the $n$th nilpotent product of ${\bf{Z}}_{p^{\alpha_i}} \ \ (1
\leq i \leq t)$ is the direct product of
${\bf{Z}}_{p^{\alpha_i}}$. So $G$ is a finite abelian $p$-group
of order $p^m$. Also $d_i$ will be equal to $\beta_i$ in Theorem 2.12.
  Therefore the following corollary is a consequence of Theorem 3.1.
\begin{cor} Let $G$ be an abelian group
of order $p^m$. Then $|{\mathcal
N}_{c_1,c_2,...,c_s}M(G)|=p^{\beta_m}$ if and
only if $G$ is an elementary abelian $p$-group, where
$$\beta_m=\chi_{c_s+1}(...(\chi_{c_2+1}(\chi _{c_1+1}(m)))...).$$
\end{cor}

Note that according to Corollary 2.12 the polynilpotent
multiplier of $G$ in the above result has its maximum order. So
the above corollary is a vast generalization of Theorem 2.5.\\

Finally in order to deal with a non-abelian case we present the following
theorem. This theorem is a generalization
 of Theorem 2.10.
\begin{thm} With the previous notation let $G$ be a finite $p$-group
of order $p^m$. If ${\mathcal N}_{c_1,c_2,...,c_s}M(G)=p^{\beta_m}$,
then the following statements holds.\\
(i) There is an epimorphism $ {\mathcal N}_{c_1,c_2,...,c_s}M(G)
\stackrel {\theta}{\rightarrow} {\mathcal
N}_{c_1,c_2,...,c_s}M(G/G')$ which is obtained from the
Fr\"{o}hlich sequence.\\
(ii) If $ker (\theta )=1$, then $G$ is an elementary abelian
$p$-group.
\end{thm}
\begin{proof} $(i)$ Let ${\cal V}$ be the variety of polynilpotent
groups of class row $(c_1,c_2,...,c_s),\\
 {\mathcal
N}_{c_1,c_2,...,c_s}$. By Corollary 2.12 we have $|{\mathcal
V}M(G)||V(G)|\leq p^{\beta_m},$. Also by the hypothesis  $|{\mathcal
N}_{c_1,c_2,...,c_s}M(G)|=p^{\beta_m}$. Therefore $|V(G)|=1$. Now set
$N=G'$ and consider the exact sequence $1 \rightarrow
G'\rightarrow G \rightarrow G/G' \rightarrow 1$. Since
$|V(G)|=1$, the above sequence is an  ${\cal V}$-central
extension. Therefore by Fr\"{o}hlich five-term exact sequence we have the following exact sequence:
$${\mathcal V}M(G)
\stackrel {\theta}{\rightarrow} {\mathcal V}M(G/G')\stackrel
{\beta} \rightarrow G' \stackrel {\alpha}\rightarrow G
\rightarrow G/G' \rightarrow
 1 .$$
Clearly $\alpha$ is a monomorphism and so $Im(\beta)=1$ .
 This means that $\theta $ is an epimorphism.\\
$(ii)$ Let $ker (\theta)=1$. Then $|{\mathcal V}M(G/G')|=|{\mathcal V}M(G)|=p^{\beta_m}\ \
 (*)$. Since $|G|=p^m$ then $|G/G'| \leq p^m$. Hence $|G/G'|= p^m$,
otherwise, if $|G/G'|= p^k$, where $k <m$, then $|{\mathcal V}M(G/G')|\leq |{\mathcal V}M(G/G')||
 V(G/G')|\leq p^{\beta_k}< p^{\beta_m}$, which is a contradiction
 to $(*)$. Hence  $|G/G'| = p^m$ which implies that $G$ is an
 abelian $p$-group. Now by Corollary 3.4 the result
 holds.
 \end{proof}

\section*{Acknowledgments}
The authors would like to thank the referee for his/her careful reading.



\begin{thebibliography}{99}
\bibitem{Baer}R. Baer. Endlichkeitskriterien f\"{u}r kommutatorgruppen. Math.
 Ann., 1952, 124: 161-177.
\bibitem{Berkovich}Ya. G. Berkovich. On the order of the commutator subgroup and
 Schur multiplier of a finite p-group. J. Algebra, 1991, 144: 269-272.
\bibitem{Frohlich}A. Fr\"{o}hlich. Baer invariants of algebras. Trans. Amer. Math. Soc., 1963, 109: 221-244.
\bibitem{Green}J. A. Green. On the number of automorphisms of a finite
group. Proc. Roy. Soc. London (A), 1956, 237: 574-581.
\bibitem{Hall}M. Hall. The Theory of Groups. Macmillian Company, New York, 1959.
\bibitem{Hekster}N. S. Hekster. Varieties of groups and isologism.
 J. Austral. Math. Soc. (A), 1989, 46: 22-60.
\bibitem{HokmabadiM}A. Hokmabadi, B. Mashayekhy. On nilpotent
 multipliers of some verbal products of groups. J. Algebra, 2008, 320: 3269-3277.
\bibitem{HokmabadiMM}A. Hokmabadi, B. Mashayekhy, F. Mohammadzadeh.  Polynilpotent multipliers of
 some nilpotent products of cyclic groups II. Inter. J. Math. Game Theory Algebra, to appear.
\bibitem{Jones}M. R. Jones.  Some inequalities for the multiplicator of a finite
  group. Proc. Amer. Math. Soc., 1973, 39: 450-456.
\bibitem{Karpilovsky}G. Karpilovsky.  The Schur Multiplier. London Math.
 Soc. Monographs, New Series No. 2, 1987.
\bibitem{Leedham_Green}C. R. Leedham-Green, S. McKay. Baer invariant, isologism, varietal laws and
 homology. Acta Math., 1976, 137: 99-150.
\bibitem{MashayekhyP}B. Mashayekhy, M.  Parvizi. Polynilpotent multipliers of finitely generated abelian groups.
 Inter. J. Math. Game Theory Algebra, 2006, 16:1 93-102.
\bibitem{MashayekhyS}B. Mashayekhy, M. A. Sanati. On the order of nilpotent
 multipliers of finite p-groups. Communications in Algebra, 2005, 33: 2079-2087.
\bibitem{MoghaddamM}M. R. R. Moghaddam, B. Mashayekhy. Higher Schur multiplicator of a finite
 abelian group. Algebra Colloquium, 1997, 4:3 317-322.
\bibitem{Moghaddam}M. R. R. Moghaddam. Some inequalities for the Baer invariant of a finite group.
 Bull. Iranian Math. Soc., 1981, 9:1 5-10.
\bibitem{Moghaddam2}M. R. R. Moghaddam. On the Schur-Baer property.  J. Austral. Math. Soc.
 (A), 1981, 31: 343-361.
\bibitem{Neumann}H. Neumann. Varieties of Groups.
 Springer-Verlag, Berlin, 1967.
\end{thebibliography}
\end{document}